\documentclass[a4paper,12pt]{amsart}

\usepackage{amssymb, amsmath, amsthm, mathrsfs, braket, xspace}
\usepackage[margin=1.0in]{geometry}
\numberwithin{equation}{section}
\usepackage[colorlinks=true]{hyperref}
\usepackage{aliascnt}

\usepackage{mathtools}
\mathtoolsset{showonlyrefs=true} 

\usepackage[all,2cell]{xy}
\objectmargin+{1mm}
\labelmargin+{0.8mm}
\SelectTips{cm}{12}

\usepackage{enumitem}
\setlist[enumerate]{itemsep=0pt,label=$(\mathrm{\roman*})$, topsep=5pt}
\setlist[itemize]{itemsep=0pt, topsep=5pt, labelindent=\parindent, leftmargin=*}
\setlist[description]{itemsep=0pt, topsep=5pt, leftmargin=*}

\usepackage{color}

\newcommand{\blue}{}

\newtheorem{thm}{Theorem}[section]

\newaliascnt{cor}{thm}
\newtheorem{cor}[cor]{Corollary}
\aliascntresetthe{cor}

\newaliascnt{lem}{thm}
\newtheorem{lem}[lem]{Lemma}
\aliascntresetthe{lem}

\newaliascnt{prop}{thm}
\newtheorem{prop}[prop]{Proposition}
\aliascntresetthe{prop}

\newaliascnt{conj}{thm}

\aliascntresetthe{conj}

\theoremstyle{definition}
\newaliascnt{dfn}{thm}
\newtheorem{dfn}[dfn]{Definition}
\aliascntresetthe{dfn}

\newaliascnt{rem}{thm}
\newtheorem{rem}[rem]{Remark}
\aliascntresetthe{rem}

\newaliascnt{prob}{thm}

\aliascntresetthe{prob}

\newaliascnt{ex}{thm}
\newtheorem{ex}[ex]{Example}
\aliascntresetthe{ex}

\newcommand{\Ab}{\mathrm{Ab}}
\newcommand{\AS}[1]{[\,#1\,)}

\newcommand{\ba}{\mathbf{a}}

\newcommand{\balpha}{\boldsymbol{\alpha}}

\newcommand{\Br}{\operatorname{Br}}

\newcommand{\Coker}{\operatorname{Coker}}

\newcommand{\dlog}{\operatorname{dlog}}
\newcommand{\et}{\mathrm{et}}

\newcommand{\Ext}{\operatorname{Ext}}

\newcommand{\F}{\mathbb{F}}
\newcommand{\Fp}{\F_p}

\newcommand{\Gal}{\operatorname{Gal}}
\newcommand{\Gm}{\mathbb{G}_{m}}
\newcommand{\Ga}{\mathbb{G}_{a}}

\newcommand{\isomto}{\xrightarrow{\simeq}}
\renewcommand{\Im}{\operatorname{Im}}

\newcommand{\ilim}{\varinjlim}
\newcommand{\Id}{\operatorname{id}}

\newcommand{\M}{\mathscr{M}}

\newcommand{\N}{\mathscr{N}}

\newcommand{\otimesZ}{\otimes_{\Z}}
\newcommand{\otimesM}{\overset{M}{\otimes}}

\newcommand{\PF}{$(\mathbf{PF})\,$}

\newcommand{\res}{\operatorname{res}}

\newcommand{\sep}{\mathrm{sep}} 
\newcommand{\surj}{\twoheadrightarrow}

\newcommand{\Spec}{\operatorname{Spec}}
\newcommand{\Symb}{\operatorname{Symb}}
\newcommand{\Tr}{\operatorname{Tr}}
\newcommand{\tr}{\operatorname{tr}}

\newcommand{\W}{\mathbb{W}}
\newcommand{\WS}{\W_S}
\newcommand{\WSOmega}{\WS\Omega}

\newcommand{\wt}[1]{\widetilde{#1}}

\newcommand{\Z}{\mathbb{Z}}

\title[An additive variant of the Differential symbol maps]{An additive variant of the differential symbol maps}
\author[T. Hiranouchi]{Toshiro Hiranouchi}\address[T. Hiranouchi]{Department of Basic Sciences, Graduate School of Engineering, 
Kyushu Institute of Technology, 
1-1 Sensui-cho, Tobata-ku, Kitakyushu-shi, 
Fukuoka 804-8550 JAPAN}
\email{hira@mns.kyutech.ac.jp}

\begin{document}
\pagenumbering{arabic}
\maketitle

\begin{abstract}
Our investigation focuses on an additive analogue of the Bloch-Gabber-Kato theorem 
which establishes a relation between the Milnor $K$-group of a field 
of positive characteristic 
and a Galois cohomology group of the field. 
Extending the Aritin-Schreier-Witt theory, 
we present an isomorphism from 
the Mackey product associated with the Witt group and the multiplicative groups 
to a Galois cohomology group. 
As a result, we give a new expression for the torsion subgroup of 
the Brauer group of a field, and more generally, the Kato homology groups.
\end{abstract}

\section{Introduction}
\label{sec:intro}

For an arbitrary field $k$ and any positive integer $m$ prime to the characteristic of $k$, 
the \textbf{norm residue isomorphism theorem} 
(\cite{Voe11}, see also \cite{Wei09}) 
is relating the Milnor $K$-group $K_n^M(k)$ of the field $k$  
and the Galois cohomology group $H^n(k,\Z/m\Z(n)) = H^n_{\et}(\Spec(k),\mu_m^{\otimes n})$. 
Precisely, for any $n\ge 0$, 
the \textbf{Galois symbol map} 
\begin{equation}\label{eq:gs}
	K_n^M(k)/mK_n^M(k) \isomto H^n(k,\Z/m\Z(n))
\end{equation}
is bijective. 
The Milnor $K$-group $K_n^M(k)$ can be replaced by 
the \textbf{Somekawa $K$-group} associated to the multiplicative groups $\Gm$
by using the canonical isomorphism 
\begin{equation}
	\label{thm:Som}
	K(k;\overbrace{\Gm,\ldots,\Gm}^n)\simeq K^M_n(k)
\end{equation}
(\cite[Thm.~1.4]{Som90}). 
The aim of this note is to study the ``additive analogue'' of the norm residue isomorphism theorem 
by replacing the first $\Gm$ in the Somekawa $K$-group $K(k;\Gm,\ldots,\Gm)$ 
with the additive group $\Ga$ or the Witt group $W_r$ more generally. 
For the case $n=1$, 
the Galois symbol map is the Kummer theory 
\[
k^{\times}/(k^{\times})^m \xrightarrow{\simeq} H^1(k,\Z/m\Z(1)).
\]
The additive analogue should be the Artin-Schreier-Witt theory
\begin{equation}\label{eq:ASW}
	W_r(k)/\wp(W_r(k)) \xrightarrow{\simeq} H^1(k,\Z/p^r\Z)
\end{equation}
for a field $k$ of characteristic $p>0$, 
where $\wp((x_0,\ldots,x_{r-1})) = (x_0^p,\ldots,x_{r-1}^p) - (x_0,\ldots,x_{r-1})$. 
Unfortunately, it has not been given that 
the Somekawa type $K$-group 
associated to $W_r$ and $\Gm$'s over an arbitrary field $k$ 
 (cf.~\cite{Hir14}, \cite{IR17}, \cite{RSY22}). 
However, after tensoring with $\Z/m\Z$, {\blue for the expression of the Milnor $K$-group in \eqref{thm:Som}}
there is no need to use the Somekawa $K$-group. 
To explain this fact, 
we recall that 
the Somekawa $K$-group is defined to be the quotient of 
$\bigoplus_{K/k:\,\mathrm{finite}} (K^{\times})^{\otimesZ n}$ 
by two kind of relations, resembling the projection formula 
and the Weil reciprocity law of the Milnor $K$-groups: 
\[
K(k;\overbrace{\Gm,\ldots,\Gm}^n) := \left.\Bigg(\bigoplus_{K/k:\,\mathrm{finite}} (K^{\times})^{\otimesZ n}\Bigg)\middle/\ \begin{matrix}(\text{projection formula}) \\ \&\ (\text{Weil reciprocity})\end{matrix}\right..
\]
The \textbf{Mackey product} $\Gm^{\otimesM n}(k) = (\overbrace{\Gm\otimesM\cdots\otimesM\Gm}^n)(k)$ of $\Gm$'s defined upon the same generators as the Somekawa $K$-group 
with the relation concerning the projection formula only:  
\begin{equation}
\label{def:Gmn} 	
\Gm^{\otimesM n}(k) := \left.\Bigg(\bigoplus_{K/k:\,\mathrm{finite}} (K^{\times})^{\otimesZ n}\Bigg)\middle/\ (\text{projection formula})\right..
\end{equation}
(for the precise definition, 
see \autoref{def:otimesM}).
By the very definition, 
the Somekawa $K$-group is a quotient of the Mackey product
$\Gm^{\otimesM n}(k)\surj K(k;\Gm,\ldots,\Gm)$. 
This quotient map 
becomes bijective  
\begin{equation}\label{eq:Kahn}
	\Gm^{\otimesM n}(k)/m  \xrightarrow{\simeq} K(k;\overbrace{\Gm,\ldots,\Gm}^n)/m 
\end{equation}
after considering modulo $m$ (Kahn's theorem, cf.~\autoref{thm:Kahn}).  
It is important to note here that the Mackey product  
$(W_r \otimesM \Gm^{\otimesM (n-1)}) (k)$ is defined 
for an arbitrary field $k$ in the same way as in \eqref{def:Gmn}. 
Using this,  
we present an isomorphism below 
which extends the Artin-Schreier-Witt theory \eqref{eq:ASW}.

\begin{thm}[\autoref{thm:bij}]\label{thm:main}
	Let $k$ be an arbitrary field of characteristic $p>0$.  
	For $r\ge 1$ and $n\ge 1$, 
	there is an isomorphism
	\[
	(W_r \otimesM \Gm^{\otimesM (n-1)}) (k)/\wp \isomto H^n(k,\Z/p^r\Z(n-1)) = H^{1}_{\et}(\Spec(k), W_r\Omega_{\log}^{n-1}),
	\]
	where 
	$(W_r \otimesM \Gm^{\otimesM (n-1)}) (k)/\wp$ 
	is the cokernel of a map 
	$\wp: (W_r \otimesM \Gm^{\otimesM (n-1)}) (k)\to (W_r \otimesM \Gm^{\otimesM (n-1)}) (k)$ 
	defined by $\wp:W_r(K)\to W_r(K)$ on the Witt groups used in \eqref{eq:ASW}.
\end{thm}

The above theorem should be compared to 
the Bloch-Gabber-Kato theorem which is  
a \emph{$p$-analogue} of the norm residue isomorphism theorem.

\begin{thm}[{Bloch-Gabber-Kato, \cite[Cor.~2.8]{BK86}}]\label{thm:BKG}
	Let $k$ be an arbitrary field of characteristic $p>0$.  
	For any $r \ge 1$ and $n \ge 0$, the \textbf{differential symbol map} $($cf.~\eqref{eq:ds}$)$
	\[
	K_n^M(k)/p^rK_n^M(k) \isomto H^n(k,\Z/p^r\Z(n)) = W_r\Omega_{k,\log}^n
	\]
	is an isomorphism.
\end{thm}

Put $H^n_{p^r}(k) := H^n(k,\Z/p^r\Z(n-1))$. 
In the case $n=2$, 
we have an isomorphism $H^2_{p^r}(k) \simeq \Br(k)[p^r]$, 
where   $\Br(k)[p^r]$ is the $p^r$-torsion part 
 of the Brauer group $\Br(k)$. 
Our theorem (\autoref{thm:main}) above gives 
an expression for the Brauer group $\Br(k)[p^r]$. 
More generally, 
for an excellent scheme $X$ over a field $k$ of characteristic $p>0$ with $[k:k^p]\le p^n$ for some $n$,  
there is a homological complex $KC^n_{\bullet}(X,\Z/p^r\Z)$ of Bloch-Ogus type: 
\[
 \cdots \to \bigoplus_{x\in X_{j}}H^{n+j+1}_{p^r}(k(x))\to \cdots  
\to \bigoplus_{x\in X_{1}}H^{n+2}_{p^r}(k(x)) \to \bigoplus_{x\in X_{0}}H^{n+1}_{p^r}(k(x)), 
\]
where $X_j$ is the set of points $x$ in $X$ with $\dim(\overline{\set{x}}) = j$ 
and $k(x)$ is the residue field at $x$ (\cite{Kat86}). 
Here, the term $\bigoplus_{x\in X_j}H_{p^r}^{n+j+1}(k(x))$ is placed in degree $j$. 
The \textbf{Kato homology group} of $X$ (with coefficients in $\Z/p^r\Z$) is defined 
to be the homology group $KH_j(X,\Z/p^r\Z) := H_j(KC_{\bullet}^n(X,\Z/p^r\Z))$.
\autoref{thm:main} gives an alternative expression of $KC^n_{\bullet}(X,\Z/p^r\Z)$
(for recent progress on the Kato homology groups, see \cite{Sai10}).

\subsection*{Notation}
Throughout this note, 
for an abelian group $G$ and $m\in \Z_{\ge 1}$,  
we write $G[m]$ and $G/m$ for the kernel and cokernel 
of the multiplication by $m$ on $G$ respectively.

{\blue
\subsection*{Acknowledgements} 
The author was supported by JSPS KAKENHI Grant Number 24K06672.
The authors would like to thank the referee for pointing out some errors in the earlier versions of this paper and suggesting remedy.
}

\section{Mackey functors}\label{sec:Mack}
In this section, we recall the notion of Mackey functors, and their product 
following \cite[Sect.~5]{Kah92a} (see also \cite[Rem.~1.3.3]{IR17}, \cite[Sect.~3]{RS00}).
Throughout this section, let $k$ be a field, 
and $\Ext_k$ the category of field extensions of $k$. 

\begin{dfn}\label{def:Mack}
	A \textbf{Mackey functor} $\M$ over a field $k$ 
	(a cohomological finite Mackey functor over $k$ in the sense of \cite{Kah92a})
    is 
    a covariant functor $\M:\Ext_k\to \Ab$ from $\Ext_k$ to the category of abelian groups $\Ab$ 
    equipped with a contravariant structure for finite extensions $L/K$ in $\Ext_k$
    satisfying the following conditions: 
    For any finite field extension $L/K$ in $\Ext_k$ and a field extension $K'/K$ in $\Ext_k$, 
	the following diagram is commutative: 
    \[
      \xymatrix@C=20mm{
      \M(L) \ar[r]^-{{\bigoplus e_i \res_{L'_i/L}}}\ar[d]_{\tr_{L/K}} & \displaystyle{\bigoplus_{i=1}^n \M(L_i')} \ar[d]^{\sum \tr_{L_i'/K'}} \\
      \M(K) \ar[r]^-{\res_{K'/K}} & \M(K'),
    }
    \]
    where 
	$L\otimes_KK' = \bigoplus_{i=1}^n A_i$ for some local Artinian algebras $A_i$  
    of dimension $e_i$ over the residue field $L_i' := A_i/\mathfrak{m}_{A_i}$.
    Here, for an extension $L/K$ in $\Ext_k$, the map 
    $\res_{L/K}\colon \M(K)\to \M(L)$ is defined by the covariant structure of $\M$ 
    and the contravariant structure gives $\tr_{L/K}\colon \M(K)\to \M(L)$ if $[L:K]<\infty$. 
 
\end{dfn}


%

The category of Mackey functors  
forms a Grothendieck abelian category (cf.~\cite[Appendix A]{KY13}) 
and hence any morphism of Mackey functor $f\colon \M\to \N$, 
that is, a natural transformation, 
gives the image $\Im(f)$, the cokernel $\Coker(f)$ and so on.

\begin{ex}\label{ex:MFexs}
	\begin{enumerate}
	\item 
	For any {\blue endomorphism $f\colon \M\to \M$ of a Mackey functor $\M$}, 
	we denote the cokernel by 
	$\M/f := \Coker(f)$. 
	This Mackey functor is given by 
	\[
	(\M/f)(K) = \Coker\left(f(K)\colon \M(K)\to \M(K)\right).
	\]

	\item 
	A commutative algebraic group $G$ over $k$ 
	induces a Mackey functor by defining $K/k\mapsto G(K)$ for any field extension $K/k$
	 (cf.\ \cite[(1.3)]{Som90}, \cite[Prop.~2.2.2]{IR17}). 
	In particular, 
	the multiplicative group $\Gm$ is a Mackey functor given by 
	$\Gm(K) = K^{\times}$ 
	for any field extension $K/k$. 
	The translation maps are the norm 
	$N_{L/K}\colon L^{\times} \to K^{\times}$ and the inclusion $K^{\times} \to L^{\times}$. 
	For the additive group $\Ga$, 
	the translation maps are 
	the trace map $\Tr_{L/K}:L\to K$ and the inclusion $K\hookrightarrow L$. 
	
	\item 
	Recall that, for $n\ge 1$, the Milnor $K$-group $K_n^M(k)$ of a field $k$   
	is the quotient group of $(k^{\times})^{\otimesZ n}$ by the subgroup generated by all elements of the form 
	$a_1\otimes \cdots \otimes a_n$ 
	with $a_i + a_{j} = 1$ for some $i\neq j$. 
	We also put $K_0^M(k) = \Z$. 
	The class of $a_1\otimes \cdots \otimes a_n $ in $K_n^M(k)$ 
	is denoted by $\set{a_1,\ldots, a_n}_{k}$. 
	For an extension $L/K$ in $\Ext_k$, 
	the restriction map $\res_{L/K}:K_n^M(K)\to K_n^M(L)$ 
	and the norm map $N_{L/K}:K_n^M(L)\to K_n^M(K)$ when $[L:K]<\infty$ 
	gives the structure of the Mackey functor  
	$K_n^M$.
	\end{enumerate}
\end{ex}

\begin{dfn}\label{def:otimesM}
	For Mackey functors $\M_1,\ldots, \M_n$ over $k$, the \textbf{Mackey product} $\M_1\otimesM \cdots \otimesM \M_n$ 
	is defined as follows: For any field extension $k'/k$, 
	\begin{equation}
		\label{eq:otimesM}
		(\M_1 \otimesM \cdots \otimesM \M_n ) (k') := 
		\left.\left(\bigoplus_{K/k':\,\mathrm{finite}} \M_1(K) \otimesZ \cdots \otimesZ \M_n(K)\right)\middle/\ (\textbf{PF}),\right. 
	\end{equation}
	where \PF\ stands for the subgroup generated 
	by elements of the following form: 
	\begin{itemize}[leftmargin=!,  align=left]
		\item [\PF]
	For a finite field extension $L/K$ in $\Ext_{k'}$, 
	\[
		 x_1 \otimes \cdots\otimes \tr_{L/K}(\xi_{i_0}) \otimes \cdots \otimes x_{n} - \res_{L/K}(x_1) \otimes \cdots \otimes \xi_{i_0} \otimes \cdots \otimes \res_{L/K}(x_n) \quad 
	\]
	for $\xi_{i_0} \in \M_{i_0}(L)$ and $x_{i} \in \M_i(K)$ for $i\neq i_0$.
	\end{itemize}
\end{dfn}

For the Mackey product  $\M_1\otimesM \cdots \otimesM \M_n$,  
we write $\set{x_1,\ldots ,x_n}_{K/k'}$ 
for the image of $x_1 \otimes \cdots \otimes x_n \in \M_1(K) \otimesZ \cdots \otimesZ \M_n(K)$ in the product 
$(\M_1\otimesM \cdots \otimesM \M_n)(k')$. 
The translation maps
$\tr_{K'/K}\colon  (\M_1\otimesM \cdots \otimesM \M_n)(K) \to (\M_1\otimesM \cdots \otimesM \M_n)(K')$ and 
$\res_{K'/K}\colon (\M_1\otimesM \cdots \otimesM \M_n)(K) \to (\M_1\otimesM \cdots \otimesM \M_n)(K')$ 
are defined as follows (cf.~\cite[Lem.~4.1.2]{IR17}, \cite[Sect.~3.4]{Akh00}):
\begin{align*}
	\tr_{K'/K}(\set{x_1,\ldots ,x_n}_{L'/K'}) &= \set{x_1,\ldots ,x_n}_{L'/K},\quad \mbox{and}\\
	\res_{K'/K}(\set{x_1,\ldots ,x_n}_{L/K})&= \sum_{i=1}^n e_i\set{\res_{L_i'/L}(x_1), \ldots , \res_{L_i'/L}(x_n)}_{L_i'/K'}, 
\end{align*}
where $L\otimes_KK' = \bigoplus_{i=1}^n A_i$ for some local Artinian algebras $A_i$  
of dimension $e_i$ over the residue field $L_i' := A_i/\mathfrak{m}_{A_i}$.
	
\begin{rem}\label{rem:tensor}
	The Mackey product $\M_1\otimesM\cdots \otimesM \M_n$ is a Mackey functor 
	with translation maps defined above (\cite[Lem.~4.1.2]{IR17}). 
	For Mackey functors $\M_1,\M_2,\M_3$ over $k$, 
    $\M_1\otimesM (\M_2\oplus \M_3) \simeq (\M_1\otimesM \M_2)\oplus (\M_1\otimesM \M_3)$ 
    by the same proof of \cite[Prop.~3.5.6]{Akh00} (Note that the Somekawa $K$-group used in \cite{Akh00} is a quotient of the Mackey product).
	Moreover, 
	we have 
	\[
	 \M_1\otimesM\M_2\otimesM\M_3 \simeq (\M_1\otimesM \M_2)\otimesM \M_3 \simeq \M_1\otimesM(\M_2\otimesM \M_3).
	\]
	\end{rem}

\section{de Rham-Witt complex}
In this section, 
we recall the notion of the de Rham-Witt complex 
and its properties following \cite{Rue07}. 
Let $k$ be a field of characteristic $p\ge 0$. 
A \textbf{truncation set} $S \subset \mathbb{N} := \Z_{>0}$ is 
a nonempty subset such that if $s \in S$ and $t\mid s$ then $t\in S$. 
For each truncation set $S$, 
let $\WS(k)$ be the ring of Witt vectors (cf.~\cite[Appendix]{Rue07}). 
The trace map $\Tr_{L/K}:\WS(L)\to \WS(K)$ 
and the restriction 
$\res_{L/K}:\WS(K)\to \WS(L)$ make the structure 
of the Mackey functor on $\WS$ (\cite[Prop.~A.9]{Rue07}). 
For a \emph{finite} truncation set $S$, 
the de Rham-Witt complex $\W_S\Omega_k^{\bullet}$ 
is a quotient of the differential graded algebra $\Omega_{\W_S(k)/\Z}^{\bullet}$ 
satisfying 
\begin{equation}\label{eq:Kaehler}
	\W_{\set{1}}\Omega_k^{\bullet}\simeq \Omega_{k/\Z}^{\bullet},\quad \mbox{and}\quad \W_S\Omega_k^0 \simeq \W_S(k),
\end{equation}
and has some natural maps 
\[
	V_n\colon\W_{S/n}\Omega_k^\bullet\to \WSOmega_k^\bullet,\quad \mbox{and}\quad F_n\colon\WSOmega_k^{\bullet}\to \W_{S/n}\Omega^{\bullet}_k, 
\]
where $S/n := \set{s \in \mathbb{N} | ns \in S}$ (cf.~\cite[Prop.~1.2]{Rue07}). 
For a truncation set $T\subset S$, 
the restriction map $\WS(k)\to \W_T(k); \ba = (a_s)_{s\in S}\mapsto \ba|_T := (a_s)_{s \in T}$ 
induces a map of differential graded algebras 
\begin{equation}
\label{eq:RST}	
R^S_T\colon \W_S\Omega_k^{\bullet} \to \W_T\Omega_k^\bullet. 
\end{equation}
For an arbitrary truncation set $S$, 
the restriction \eqref{eq:RST} gives 
the inverse system $(\W_{S_0}\Omega_k^{\bullet})_{S_0\subset S}$, 
where $S_0$
runs through the set of all finite truncation set contained in $S$.
We define 
\begin{equation}\label{eq:WSO}	
	\WSOmega_k^{\bullet} := \varprojlim_{S_0\subset S}\W_{S_0}\Omega_k^{\bullet}.
\end{equation}

For any extension $L/K$ in the category $\Ext_k$,  
the natural homomorphism $\WS(K) \to \WS(L)$ of rings (cf.~\cite[A.1]{Rue07}) 
induces $\Omega_{\WS(K)}^{\bullet}\to \Omega_{\WS(L)}^{\bullet}$ 
and hence the restriction map 
\begin{equation}\label{eq:resWSO}
	\res_{L/K}\colon \WSOmega_K^\bullet\to \WSOmega_L^\bullet.
\end{equation}
Define $pS := S\cup \set{ps | s\in S}$. 
There is a map of differential graded algebras 
\begin{equation}
	\label{eq:pbar}
	\underline{p}\colon\WSOmega_K^{\bullet}\to \W_{pS}\Omega_L^{\bullet}; \omega \mapsto p\wt{\omega},
\end{equation}
where $\wt{\omega}$ is a lift of $\omega$ to $\W_{pS}\Omega_L^{\bullet}$ 
with respect to the map $R_{S}^{pS}$ \eqref{eq:RST} (\cite[Def.~2.5]{Rue07}). 
For a finite truncation set $S$, 
the correspondence $K/k\mapsto \WSOmega^n_K$ forms a Mackey functor  
using the trace map below (\cite[Prop.~2.7]{Rue07}). 

\begin{thm}[{\cite[Thm.~2.6]{Rue07}, \cite[Thm.~{\blue A}.2]{KP21}}]\label{thm:Tr}
	{\blue Let $L/K$ be a finite field extension in $\Ext_k$ and $S$ a finite truncation set.}	
	Then, there is a map of differential graded $\WSOmega_K^{\bullet}$-modules 
	\[
	\Tr_{L/K}\colon \WSOmega_L^{\bullet}\to \WSOmega_K^{\bullet}
	\]
	satisfying the following properties: 
	
	\begin{enumerate}[label=$(\mathrm{\alph*})$]
	
	\item If the extension $L/K$ is separable, then 
{\blue the isomorphism 
	$\WS(L)\otimes_{\WS(K)} \WS\Omega_K^n \xrightarrow{\simeq} \WSOmega_L^n; \balpha \otimes \omega \mapsto \balpha\omega$} gives 
	the map $\Tr_{L/K}$   
	by  
	\[
	 {\blue \WSOmega_L^n \xleftarrow{\simeq}\WS(L)\otimes\WSOmega_K^n \xrightarrow{\Tr_{L/K}\otimes\Id} \WS(K)\otimes_{\WS(K)}\WSOmega^n_K   \simeq \WSOmega^n_{K}.}
	\] 
	
	\item If $L/K$ is purely inseparable of degree $p$ and 
	$\omega \in \WSOmega_L^n$, 
	then $\Tr_{L/K}(\omega) = R^{pS}_S(\eta)$, 
	where $\res_{L/K}(\eta) = \underline{p}(\omega)$ 
	for some $\eta \in \W_{pS}\Omega_K^n$.
	\item If $K\subset M \subset L$ are finite field extensions, then we have 
	$\Tr_{L/K} = \Tr_{M/K}\circ \Tr_{L/M}$. 
	\item 
	The map $\Tr_{L/K}$ commutes with $V_n,F_n$ and the restriction $R_{T}^S$.
	\end{enumerate}
\end{thm}

In the case where $k$ has \emph{positive characteristic $p>0$}, 
taking the truncation set $P_r := \set{1,p,p^2,\ldots,p^{r-1}}$ for $r\ge 1$, 
we define $W_{r}\Omega_k^{\bullet} := \W_{P_r}\Omega_k^{\bullet}$. 
This is the classical $p$-typical de Rham-Witt complex of $k$ (\cite{Ill79}). 
{\blue In particular, $W_{r}\Omega_k^0 = W_r(k)$ is the ring of $p$-typical Witt vectors of length $r$ in $k$.
It is customary to write  
\begin{align}\label{eq:RSTp}
	V &:= V_p\colon W_{r}\Omega_k^{\bullet} \to W_{r+1}\Omega_k^{\bullet},\notag \\ 
	F &:= F_p\colon W_{r+1}\Omega_k^{\bullet} \to W_{r}\Omega_k^{\bullet},\quad  \mbox{and}\\
	R &:= R^{P_{r+1}}_{P_{r}}\colon W_{r+1}\Omega_k^{\bullet} \to W_{r}\Omega_k^{\bullet} \notag. 
\end{align}
Furthermore, the map $F\colon W_{r+1}\Omega_k^{\bullet} \to W_{r}\Omega_k^{\bullet}$ induces 
\begin{equation}
	\label{eq:F}
	F\colon W_r\Omega_k^n \to W_r\Omega_k^n/\mathrm{d} V^{r-1}\Omega_k^{n-1}
\end{equation}
(\cite[Prop.~3.3]{Ill79}, \cite[Lem.~2.7]{Shi07}).
}
By \cite[Thm.~1.11]{Rue07}, for a finite truncation set $S$, 
there is a canonical decomposition (with respect to $S$) 
\begin{equation}\label{eq:dec}
	\WSOmega_k^n \simeq \prod_{(m,p) = 1}W_{r_m}\Omega_k^n,
\end{equation}
where $r_m := \# (S/m \cap \set{1,p,p^2,\ldots }) +1$. 
There is a natural homomorphism 
\[
\dlog\colon k^{\times} \to W_r\Omega_k^1; b\mapsto \dlog[b] := \mathrm{d}[b]/[b],
\]
where $[b] := (b,0,\ldots, 0) \in W_r\Omega_k^1$. 

\begin{lem}\label{lem:NTr}
For any $r\ge 1$, and for any finite field extension $L/K$ in $\Ext_k$, 
the following diagram is commutative: 
\[
\xymatrix{
L^{\times} \ar[r]^-{\dlog }\ar[d]_{N_{L/K}} & W_r\Omega_L^1 \ar[d]^{\Tr_{L/K}}\, \\ 
K^{\times} \ar[r]^-{\dlog } & W_r\Omega_K^1.
}
\]	
\end{lem}
\begin{proof}
	For any $\beta\in L^{\times}$, 
	we show the equality 
	\begin{equation}\label{eq:NTr}
		\Tr_{L/K}(\dlog[\beta]) = \dlog[N_{L/K}\beta].
	\end{equation}
	By taking the separable closure of $L/K$, 
	the extension $L/K$ is written as a finite separable extension 
	followed by a purely inseparable extension. 
	Since every finite purely inseparable extension is 
	decomposed as a sequence of extensions of prime degree $p$, 
	it suffices to show the assertion in the cases where 
	the extension $L/K$ is either separable or 
	purely inseparable of degree $p$. 
	In the former case, the equality \eqref{eq:NTr} follows from the construction of the trace map (\cite[Proof of Thm.~2.6]{Rue07}).
	Next, we assume that the extension $L/K$ is purely inseparable of degree $p$. 
	Denoting the Teichm\"uller map as 
	$[-]_{r,K}: K^{\times} \to W_r\Omega_K^1$ in this proof, 
	we obtain 
	\begin{align*}
		\underline{p}(\dlog([\beta]_{r,L})) &= p \dlog( [\beta]_{r+1,L}) \\
		&= \dlog([\beta^p]_{r+1,L})\\
		&= \res_{L/K}(\dlog([N_{L/K}\beta]_{r+1,K}))\quad (\mbox{by $N_{L/K}(\beta) = \beta^p$)}.
	\end{align*}
	By \autoref{thm:Tr} (b), 
	we have
	\[
		\Tr_{L/K}(\dlog([\beta]_{r,L})) = R(\dlog([N_{L/K}\beta]_{r+1,K})) = \dlog([N_{L/K}\beta]_{r,K}),
	\]
	where $R = R^{P_{r+1}}_{P_{r}}$ (cf.~\eqref{eq:RSTp}).
\end{proof}

The subgroup of $W_r\Omega_k^{n}$ generated by 
$\dlog [b_1] \cdots \dlog [b_n]$ for $b_1,\ldots,b_n\in k^{\times}$ 
is denoted by $W_r\Omega_{k,\log}^n$. 
We put 
\begin{equation}\label{def:H}
	H^n(k,\Z/p^r(m)) = H^{n-m}(k,W_r\Omega_{k^{\sep},\log}^m). 
\end{equation}
For any $r \in \Z_{>0}$ and $n \in \Z_{\ge 0}$, the differential symbol map 
\begin{equation}\label{eq:ds}	
s_{k,p^r}^n:K_n^M(k)/p^r \isomto H^n(k,\Z/p^r(n)) = W_r\Omega_{k,\log}^n
\end{equation}
introduced in \autoref{thm:BKG} 
is given by $\set{a_1,\ldots,a_n}_k\mapsto \dlog[a_1]\cdots \dlog[a_n]$.

\section{Mackey products and the Milnor $K$-groups}
\label{sec:Kummer}
Let $k$ be a field of characteristic $p\ge 0$. 
As referred in Introduction, 
B.~Kahn gives an isomorphism 
$\Gm^{\otimesM n}(k)/m \simeq K_n^M(k)/m$ 
between the Mackey product of $\Gm$'s (\autoref{def:otimesM}) and 
the Milnor $K$-group  modulo $m$ 
(\autoref{ex:MFexs} (iii)) 
which is announced in \cite[Rem.~4.2.5 (b)]{RS00} without proof. 
In this section, 
we give a proof of this theorem (\autoref{thm:Kahn}).
 
\begin{prop}
\label{prop:surjm}
	For any $n\ge 0$, and any $m\ge 1$, 
	there is a surjective homomorphism
	$\pi_{k,m}^n:\Gm^{\otimesM n}(k)/m \surj K_n^M(k)/m$.
\end{prop}
\begin{proof}
	By \cite[Thm.~1.4]{Som90}, we have a canonical isomorphism 
	$K(k;\Gm,\ldots ,\Gm) \simeq K_n^M(k)$. 
	As the Somekawa $K$-group is a quotient of the Mackey product $\Gm^{\otimesM n}(k)$, 
	there is a surjective homomorphism 
	$\pi_{k,m}^n:\Gm^{\otimesM n}(k)/m \surj K_n^M(k)/m$
	which is given by $\pi_{k,m}(\set{a_1,\ldots,a_n}_{K/k}) = N_{K/k}(\set{a_1,\ldots,a_n}_{K})$, 
	where $N_{K/k}\colon K_n^M(K)\to K_n^M(k)$ is the norm map on the Milnor $K$-groups.
\end{proof}

For simplicity, we often abbreviate 
$\res_{L/K}(x)$ in $\Gm(L)/m=L^{\times}/m$ as $x$ 
for $x \in \Gm(K)/m = K^{\times}/m$. 
For a fixed prime number $l$, 
we denote by $\Symb_l(k)$ the subgroup of $\Gm^{\otimesM n}(k)/l$ 
generated by the symbols of the form $\set{a_1,\ldots,a_n}_{k/k}$ for $a_i \in k^{\times}$. 

\begin{lem}\label{lem:SymbK}
	For a prime number $l$, 
	the map $\pi_{k,l}^n\colon \Gm^{\otimesM n}(k)/l \to K_n^M(k)/l$ induces an isomorphism
	$\Symb_l(k)\xrightarrow{\simeq} K_n^M(k)/l$.
\end{lem}
\begin{proof}
	To show that 
	$\varphi_k: K_n^M(k)/l \to \Gm^{\otimesM n}(k)/l;  
	\set{a_1,\ldots,a_n}_k \mapsto \set{a_1,\ldots ,a_n}_{k/k}$
	is well-defined, 
	we prove $\set{a,1-a, a_2,\ldots,a_n}_{k/k} = 0$ in $\Gm^{\otimesM n}(k)/l$ 
	for some $a \not\in (k^{\times})^l$.  
	First, we consider the case $l\neq p$ (this is a direct consequence of \cite[Lem.~4.2.6]{RS00}).
	Let $T^l - a = \prod_{i}f_i(T)$ 
	be the decomposition with monic and irreducible polynomials $f_i(T)$ in $k[T]$. 
	For each $i$, let $\alpha_i \in k^{\sep}$ be a root of $f_i(T)$ 
	and put $K_i = k(\alpha_i)$. 
	Then, 
	\[
	1-a = \prod_{i} f_i(1) = \prod_i N_{K_i/k}(1 - \alpha_i).
	\]
	We have
	\begin{align*}
	&\set{a,1-a,a_2,\ldots,a_n}_{k/k} \\
	&= 
	\sum_i \set{a,N_{K_i/k}(\alpha_i),a_2,\ldots,a_n}_{k/k} \\
	&= 	\sum_i \set{a,\alpha_i, a_2,\ldots, a_n}_{K_i/k}\ (\mbox{by \PF })\\
	&= \sum_i \set{(\alpha_i)^l,\alpha_i,  a_2,\ldots, a_n}_{K_i/k}\quad (\mbox{by $(\alpha_i)^l - a = \prod_{j}f_j(\alpha_i)=0$})\\
	&= 0 \quad \mbox{in $\Gm^{\otimesM n}(k)/l$}.
	\end{align*} 
	
	Next, we treat the case $p>0$ and $l=p$. 
	Let us consider the purely inseparable extension $K := k(\sqrt[p]{a})/k$ of degree $p$. 
	The norm map gives $N_{K/k}(\sqrt[p]{a}) = (\sqrt[p]{a})^p = a$. 
	Therefore, by the projection formula \PF, we have 
	\begin{align*}
		\set{a,1-a,a_2,\ldots,a_n}_{k/k} &= \set{N_{K/k}(\sqrt[p]{a}),1-a,a_2,\ldots,a_n}_{k/k} \\
		&= \set{\sqrt[p]{a},1-a,a_2,\ldots, a_n}_{K/k}.
	\end{align*}
	The equality $1-a = (1-\sqrt[p]{a})^p$ in $K$,
	gives $\set{a,1-a, a_2,\ldots,a_n}_{k/k} = 0$.
\end{proof}

\begin{thm}\label{thm:modl}
	Let $l$ be a prime number. 
	We assume that $k$ contains a primitive $l$-th root of unity when $l\neq p$. 
	Then, for any $n\ge 1$, 
	the map $\pi_{k,l}^n\colon \Gm^{\otimesM n}(k)/l \to K_n^M(k)/l$ is bijective.
\end{thm} 
\begin{proof}
	Put $\pi_k:= \pi_{k,l}^n$. 
	From \autoref{lem:SymbK}, 
	we have a map $\varphi_k:K_n^M(k)/l  \to \Gm^{\otimesM n}(k)/l$ 
	defined by 
	$\varphi_k(\set{a_1,\ldots, a_n}_k) = \set{a_1,\ldots, a_n}_{k/k}$ and this gives $\pi_k \circ \varphi_k = \Id$. 
	To show \autoref{thm:modl}, it is sufficient to 
	prove that $\varphi_k$ is surjective or $\pi_k$ is injective.
	
	Let $k^{(l)}/k$ be the field extension 
	corresponding to the $l$-Sylow subgroup of $\Gal(k^{\sep}/k)$. 
	The following diagram is commutative:
	\[
	\xymatrix{
	\Gm^{\otimesM n}(k^{(l)})/l \ar[r]^-{\pi_{k^{(l)}}} & K_n^M(k^{(l)})/l \\
	\Gm^{\otimesM n}(k)/l \ar[u]^{\res_{k^{(l)}/k}} \ar[r]^-{\pi_k} & K_n^M(k)/l
	\ar[u]_{\res_{k^{(l)}/k}}.
	}
	\]
	To show that $\pi_k$ is injective, take any  $\xi\in \Gm^{\otimesM n}(k)/l$ 
	with $\pi_k(\xi) = 0$. 
	If we assume that $\pi_{k^{(l)}}$ is bijective, 
	we have $\res_{k^{(l)}/k}(\xi) = 0$ in $\Gm^{\otimesM n}(k^{(l)})/l$. 
	One can write $\Gm^{\otimesM n}(k^{(l)})/l \simeq \ilim_{K}(\Gm^{\otimesM n}/l)(K)$, 
	where $K$ runs through all finite extension $K/k$ within $k^{(l)}$ (for the proof, see \cite[Lem.~3.6.4]{Akh00}). 
	There exists a subextension $k\subset K\subset k^{(l)}$ 
	such that 
	$[K:k]$ is finite and 
	$\res_{K/k}(\xi) = 0$. 
	This implies $[K:k]{\blue \xi} = \tr_{K/k}\circ \res_{K/k}(\xi) = 0$. 
	Since $[K:k]$ is prime to $l$, 
	the restriction $\res_{K/k}$ is injective so that we obtain $\xi = 0$. 
	Consequently, we may assume 
	$k = k^{(l)}$ and it is left to show that $\varphi_k$ is surjective. 
	Take any symbol $\set{\alpha_1,\ldots, \alpha_n}_{K/k}$ in $\Gm^{\otimesM n}(k)/l$ 
	and we show that this is in $\Symb_l(k)$. 
	There exists a tower of fields 
	$k = k_0 \subset k_1 \subset k_2 \subset \cdots \subset k_r= K
	$
	such that $[k_i:k_{i-1}] = l$. 
	By induction on $r$, 
	we may assume $[K:k] = l$. 
	It is known that $\set{\alpha_1,\ldots ,\alpha_n}_K \in K_n^M(K)$ 
	is written as a sum of elements of the form 
	$\set{\xi, x_2,\ldots, x_n}_K$ 
	for some $\xi \in K^{\times}$ and $x_2,\ldots, x_n\in k^{\times}$ 
	(\autoref{prop:BT} below). 
	From  \autoref{lem:SymbK}, we have 
	\[
	{\blue \set{\alpha_1,\ldots,\alpha_n}_{K/K} =\sum_i\set{\xi^{(i)},x_2^{(i)},\ldots,x_n^{(i)}}_{K/K}}
	\]
	for some $\xi^{(i)}\in K^{\times}$ and $x_2^{(i)},\ldots,x_n^{(i)}\in k^{\times}$. 
	By the projection formula \PF, 
	we obtain 
	\begin{align*}
	\set{\alpha_1,\ldots,\alpha_n}_{K/k} & = \tr_{K/k} (\set{\alpha_1,\ldots,\alpha_n}_{K/K})\\
	&=\sum_i \tr_{K/k}(\set{\xi^{(i)},x_2^{(i)},\ldots,x_n^{(i)}}_{K/K}) \\
	&=\sum_i \set{\xi^{(i)},x_2^{(i)},\ldots,x_n^{(i)}}_{K/k}  \\
	&= \sum_i \set{N_{K/k}(\xi^{(i)}),x_2^{(i)},\ldots,x_n^{(i)}}_{k/k}.
	\end{align*}
	Namely, $\set{\alpha_1,\ldots,\alpha_n}_{K/k}$ is in $\Symb_l(k)$ 
	and the map $\varphi_k$ is surjective. 
\end{proof}

\begin{prop}[{\cite[Chap.~I, Cor.~5.3]{BT73}, \cite[Prop.~2.2.3]{Akh00}}]\label{prop:BT}
	Let $k$ be a field and $l$ a prime number. 
	Suppose that every finite extension of $k$ is of degree $l^r$ for some $r\ge 0$. 
	Then, for an extension $K/k$ of degree $l$, 
	the group $K_n^M(K)$ is generated by symbols of the form 
	$\set{\xi, x_2,\ldots, x_n}_K$ 
	where $\xi \in K^{\times}$ and $x_2,\ldots, x_n \in k^{\times}$. 
\end{prop}

\begin{thm}
\label{thm:Kahn}
	Let $k$ be an arbitrary field of characteristic $p\ge 0$. 
	For any $m\ge 1$ and $n\ge 0$, the map 
	$\pi_{k,m}^n \colon \Gm^{\otimesM n}(k)/m \isomto K_n^M(k)/m$
	is bijective. 
	Moreover, this gives an isomorphism of Mackey functors
	$\Gm^{\otimesM n}/m\simeq K_n^M/m$.
\end{thm}
\begin{proof}
	Putting $\Gm^{\otimesM0}(k) = \Z$, 
	for the cases $n=0$ and $n=1$, 
	there is nothing to show so that we assume $n\ge 2$. 
	Put $m = m'p^r$ with an integer $m'>0$ prime to $p$.
	The Galois symbol map \eqref{eq:gs} and the differential symbol map (\autoref{thm:BKG}) give an isomorphism 
	\[
	s_{k,m}^n: K_n^M(k)/m \xrightarrow{\simeq} H^n(k,\Z/m'(n))\oplus H^n(k,\Z/p^r(n)). 
	\]
	In the following, we often identify $K_n^M(k)/m$ and $H^n(k,\Z/m(n))$ through this isomorphism. 
	By \autoref{prop:surjm}, $\pi_{k,m}^n$ is surjective 
	which is given by $\pi_{k,m}(\set{a_1,\ldots,a_n}_{K/k}) = N_{K/k}(\set{a_1,\ldots,a_n}_{K})$.
	It is enough to show that $\pi_{k,l^r}^n:\Gm^{\otimesM n}(k)/l^r \surj K_n^M(k)/l^r$ 
	is injective for any prime $l$ and $r\ge 1$. 
	Now, we divide the proof into the two cases (a) $l\neq p$, and (b) $l = p>0$. 
	
	In the case (a): $l \neq p$, 
	by the standard arguments, we will reduce to showing 
	the case 
	that $k$ contains a primitive $l$-th root of unity  
	(e.g., \cite[Sect.~1.2]{Kah97}, \cite[Sect.~1]{Sus99}). 
	Take a primitive $l$-th root of unity $\zeta_l \in k^{\sep}$ and put $k' := k(\zeta_l)$ which is a field extension of $k$ of degree prime to $l$. 
	There is a commutative diagram:
	\[
	\xymatrix{
	\Gm^{\otimesM n}(k')/l^r \ar[r]^-{\pi_{k',l^r}^n} & K_n^M(k')/l^r\\
	\Gm^{\otimesM n}(k)/l^r \ar[u]^{\res_{k'/k}}\ar[r]^-{\pi_{k,l^r}^n} & K_n^M(k)/l^r\ar[u]_{\res_{k'/k}}.
	}
	\]
	Since $[k':k]$ is prime to $l$, 
	the restriction maps are injective. 
	If $\pi_{k',l^r}^n$ is injective, so is $\pi_{k,l^r}^n$. 
	Accordingly, we may assume $\zeta_l\in k$. 
	Next, we 
	consider the commutative diagram below: 
	\[
	\xymatrix@C=5mm{
	\Gm^{\otimesM (n-1)}(k)/l \ar[r]^-{\set{\zeta_l,-,\ldots,-}_{k/k}} \ar[d]^{\pi_{k,l}^{n-1}} & \Gm^{\otimesM n}(k)/l^r \ar[r]\ar[d]^{\pi_{k,l^{r}}^n}&  \Gm^{\otimesM n}(k)/l^{r+1} \ar[r]\ar[d]^{\pi_{k,l^{r+1}}^n} & \Gm^{\otimesM n}(k)/l\ar[d]^{\pi_{k,l}^n}\ar[r] & 0\\
	H^{n-1}(k,\Z/l(n-1)) \ar[r]^-{\zeta_l \cup -} & H^n(k,\Z/l^r(n))\ar[r] & H^n(k,\Z/l^{r+1}(n))\ar[r] & H^n(k,\Z/l(n)).
	}
	\]
	In the above diagram, 
	the lower sequence is a part of the long exact sequence 
	induced from the short exact sequence 
	$0 \to \mu_{l^{r}}^{\otimes n} \to \mu_{l^{r+1}}^{\otimes n} \to \mu_l^{\otimes n} \to 0$. 
	The upper one is a complex and exact except the term $\Gm^{\otimesM n}(k)/l^r$. 
	By the induction on $r$ and \autoref{thm:modl}, 
	$\pi_{k,l^{r+1}}^n$ is bijective. 
	
	Consider the case (b): $l = p>0$. 
	As in the proof of \cite[Cor.~2.8]{BK86}, we have the following commutative diagram 
	with exact rows: 
	\begin{equation}
	\label{diag:BGK}
		\vcenter{
	\xymatrix{
	& \Gm^{\otimesM n}(k)/p \ar[r]\ar@{->>}[d]^{\pi_{k,p}^n} & \Gm^{\otimesM n}(k)/p^{r+1} \ar[r]\ar@{->>}[d]^{\pi_{k,p^{r+1}}^n} & \Gm^{\otimesM n}(k)/p^{r} \ar[r]\ar@{->>}[d]^{\pi_{k,p^{r}}^n}  & 0\\
	0\ar[r] & K_n^M(k)/p \ar[r] & K_n^M(k)/p^{r+1} \ar[r] & K_n^M(k)/p^{r} \ar[r] & 0.
	}}
	\end{equation}
	In fact, the exactness of the lower horizontal sequence follows from 
	\[
	\xymatrix{
	& K_n^M(k)/p \ar[r]\ar[d]^{s_{k,p}^n} & K_n^M(k)/p^{r+1} \ar[r]\ar[d]^{s_{k,p^{r+1}}^n} & K_n^M(k)/p^{r} \ar[r]\ar[d]^{s_{k,p^{r}}^n}  & 0\\
	0\ar[r] & H^n(k,\Z/p(n)) \ar[r] & H^n(k,\Z/p^{r+1}(n)) \ar[r] & H^n(k,\Z/p^r(n)) &
	}
	\]
	and the Bloch-Gabber-Kato theorem (\autoref{thm:BKG}). 
	By the induction on $r$ and the digram \eqref{diag:BGK}, 
	\autoref{thm:Kahn} follows from \autoref{thm:modl}.
\end{proof}

\section{Mackey products and the de Rham-Witt complex}\label{sec:symbol}
Let $k$ be a field of positive characteristic $p>0$ 
and $k^{\sep}$  the separable closure of $k$ in an algebraic closure of $k$. 
For an extension $K/k$, 
we consider the $p$-th power Frobenius $\phi\colon K\to K; a\mapsto a^p$. 
{\blue The induced homomorphism $W_r(\phi) \colon W_{r+1}(k)\to W_r(k)$ 
coincides with $F\colon W_r(k) \to W_r(k)$ (cf.~\eqref{eq:F}).}
For any $r\ge 1$ and any field extension $K/k$, 
we consider the homomorphism 
\begin{equation}
	\label{eq:wp}
\wp ={\blue F}-\Id: W_r(K) \to W_r(K).
\end{equation}
The map $\wp$ also 
induces the map $\wp\colon W_r\to W_r$ of Mackey functors {\blue (\cite[Prop.~A.9 (v)]{Rue07})}. 
{\blue 
For any Mackey functor $\M$, we define a map 
\begin{equation}\label{eq:wpM}
\wp := \wp\otimesM \Id \colon W_r \otimesM \M\to W_r \otimesM \M.
\end{equation}
For simplicity, we denote also the induced map 
$\wp(k)\colon (W_r\otimesM \M) (k) \to (W_r\otimesM \M)(k)$ by $\wp$. 

\begin{lem}\label{lem:mod}
	For any Mackey functor $\M$ over $k$, 
	We have isomorphisms 
	\[
	(W_r\otimesM \M)/\wp \simeq W_r/\wp \otimesM \M \simeq W_r/\wp \otimesM \M/p^r.
	\]
\end{lem}
\begin{proof}
	Recalling from \autoref{ex:MFexs} (i), 
	$(W_r\otimesM \M)/\wp$ is 
	the cokernel of $\wp\colon W_r\otimesM \M \to W_r\otimesM \M$ defined in \eqref{eq:wpM}. 
	The isomorphism $(W_r\otimesM \M)/\wp \simeq W_r/\wp\otimesM \M$ follows from the right exactness of the product $-\otimesM \M$.
	Since the Witt ring $W_{r}(K)$ is annihilated by $p^r$,  
	we obtain 
	$W_r/\wp \otimesM \M \simeq W_r/\wp\otimesM\M/p^r$. 
\end{proof}
}

For any $n\ge 1$, the cokernel of the extended $\wp$ 
with respect to the product $\Gm^{\otimesM (n-1)}$ 
(cf.~\eqref{eq:wpM}) is denoted by 
\begin{equation}\label{def:M}
	M_{p^r}^n := (W_r\otimesM \Gm^{\otimesM (n-1)})/\wp.
\end{equation}

\begin{cor}
	If $k$ is perfect and $n\ge 2$, then 
	$M_{p^r}^n(k) = 0$.
\end{cor}
\begin{proof}
This follows from \autoref{lem:mod} and $(\Gm/p^r)(k) = 0$.
\end{proof}

We put 
\begin{equation}\label{def:Hpr}
	H^n_{p^r}(k) := H^n(k,\Z/p^r(n-1)) = H^1(k,W_r\Omega_{\log}^{n-1}).
\end{equation}
{\blue 
The map $F\colon W_r\Omega_k^{n-1}\to  W_r\Omega_k^{n-1}/\mathrm{d}V^{r-1}\Omega_k^{n-2}$ 
(cf.~\eqref{eq:F})}
%
induces the following 
short exact sequence: 
\begin{equation}\label{eq:strH^n}
	0\to H^{n-1}(k,\Z/p^r(n-1))\to W_r\Omega_k^{n-1}\xrightarrow{\wp} W_r\Omega_k^{n-1}/{\blue \mathrm{d}V^{r-1}\Omega_k^{n-2}}\to H^n_{p^r}(k)\to 0, 
\end{equation}
where $\wp:= F-\Id$
(cf.~\cite[Prop.~2.8]{Shi07}, \cite[Sect.~A.1.3]{Izh00}).
{\blue Since $F$ is a ring homomorphism of $W_r\Omega_k^{\bullet}$ 
and satisfies $F(\mathrm{d}[b]) = [b]^{p-1}\mathrm{d}[b]$,} 
the map $\wp$ is given by $\wp(\mathbf{a}\dlog[b_1]\cdots \dlog [b_{n-1}]) = \wp(\mathbf{a}) \dlog[b_1]\cdots \dlog[b_{n-1}]$. 
Moreover, {\blue we have} the following explicit description 
of $H^n_{p^r}(k)$. 

\begin{thm}[{\cite[Sect.~2, Cor.~4 to Prop.~2]{Kat80}, see also \cite[Sect.~A.1.3]{Izh00}}] 
\label{thm:Kato}
	The natural homomorphism 
	$\ba\otimes b_1\otimes \cdots \otimes b_{n-1} \mapsto \AS{\ba,b_1,\ldots ,b_{n-1}}_k := \ba \dlog[b_1]\cdots \dlog[b_{n-1}]$ 
	induces an isomorphism 
	\[
	\left(W_r(k)\otimesZ (k^{\times})^{\otimesZ (n-1)}\right)/J \isomto H^{n}_{p^r}(k),
	\]
	where 
	$J$ is the subgroup generated by all elements of the form:
	\begin{enumerate}[label=$(\mathrm{\alph*})$]
		\item $\ba\otimes b_1\otimes \cdots \otimes b_{n-1}$ with $b_i = b_j$ for some $i\neq j$.
		\item $(0,\ldots, 0, a, 0,\ldots ,0)\otimes a\otimes b_2\otimes \cdots \otimes b_{n-1}$ 
		for some $a \in k$.
		\item $\wp(\mathbf{a})\otimes b_1\otimes \cdots  \otimes b_{n-1}$.
	\end{enumerate}
\end{thm}

From the sequence \eqref{eq:strH^n}, we regard $H^n_{p^r}(k)$ 
as a quotient of $W_r\Omega_k^{n-1}$. 
{\blue For any finite field extension $K/k$, the trace map $\Tr_{K/k}:W_r\Omega_K^{\bullet}\to W_r\Omega_k^{\bullet}$ 
is a map of differential graded modules 
commutes with $V$ and $F$ (\autoref{thm:Tr}). 
The trace map  
$\Tr_{K/k}:W_r\Omega_K^{n-1}\to W_r\Omega_k^{n-1}$ 
induces a map 
$H^n_{p^r}(K)\to H^n_{p^r}(k)$ 
which we also write $\Tr_{K/k}$.
}
%
\begin{prop}\label{prop:surj}
	There is a surjective homomorphism 
	$t_{k,p^r}^n:M_{p^r}^n(k) \to H^n_{p^r}(k)$ 
	which is given by 
	\begin{equation}
	\label{def:tnpr}
		\set{\ba,b_1,\ldots,b_{n-1}}_{K/k} \mapsto \Tr_{K/k}(\AS{\mathbf{a},b_1,\ldots,b_{n-1}}_K).
	\end{equation}
\end{prop}
\begin{proof}
	We define 
	$t_{k,p^r}^n\colon(W_r \otimesM \Gm^{\otimesM (n-1)})(k)\to H^n_{p^r}(k)$  
	by the corresponding given in \eqref{def:tnpr}. 
	To show that $t_{k, p^r}^n$ is well-defined,
	take finite field extensions $k\subset K\subset L$. 

	\smallskip
	\noindent 
	Case (a):
	For $\beta_{i_0} \in \Gm(L)$,  $b_i\in \Gm(K)$ with $i\neq i_0$, and $\ba\in W_r(K)$,  
	the relation \PF implies the equality:
	\[
	\set{\ba,b_1,\ldots,N_{L/K}(\beta_{i_0}),\ldots,b_{n-1}}_{K/k} = \set{\ba, b_1,\ldots, \beta_{i_0},\ldots, b_{n-1}}_{L/k}.
	\]
	Here, we omit the restriction maps $\res_{L/K}$ in the right. 
	By the transitivity of the trace map $\Tr_{L/k} = \Tr_{K/k} \circ \Tr_{L/K}$ (\autoref{thm:Tr} (c)), 
	it is sufficient to show the equality
	\begin{equation}
		\label{eq:PFTr1}		
	\AS{\ba,b_1,\ldots,N_{L/K}(\beta_{i_0}),\ldots,b_{n-1}}_K = \Tr_{L/K}(\AS{\ba,b_1,\ldots,\beta_{i_0},\ldots,b_{n-1}}_L).
	\end{equation}
	The $W_r\Omega_K^{n-1}$-linearity of the trace map $\Tr_{L/K}$, 
	and \autoref{lem:NTr} imply the equality \eqref{eq:PFTr1}.
	
\smallskip
\noindent 
Case (b): 	
	For $\balpha\in W_r(L)$ and $b_{i} \in \Gm(K)$ for $1\le i\le n-1$, 
	the relation \PF implies the equality:
	\[
	\set{\Tr_{L/K}(\balpha),b_1,\ldots,b_{n-1}}_{K/k} = \set{\balpha, b_1,\ldots , b_{n-1}}_{L/k}.
	\]
{\blue 
	By the $W_r\Omega_K^{n-1}$-linearity of the trace map $\Tr_{L/K}$, we have}
	\begin{equation}
		\label{eq:PFTr2}		
	\AS{\Tr_{L/K}(\balpha),b_1,\ldots,b_{n-1}}_K = \Tr_{L/K}(\AS{\balpha,b_1,\ldots,b_{n-1}}_L).
	\end{equation}
	Hence, $t^n_{k,p^r}$ is well-defined.
	
	
	By \autoref{thm:Kato}, the map $t_{k,p^r}^n$ is surjective 
	and factors through 
	the quotient $M_{p^r}^n(k)$. 
	The induced homomorphism $M_{p^r}^n(k)\to H^n_{p^r}(k)$ 
	is the required one. 
\end{proof}

%
%
%
Let $\Symb_p(k)$ be the subgroup of $M_{p}^n(k) = (\Ga \otimesM\Gm^{\otimesM(n-1)})(k)/\wp$ 
	generated by the symbols of the form $\set{a,b_1,\ldots ,b_{n-1}}_{k/k}$. 
First we show that the map $t_{k,p}^n$ is bijective 
on the subgroup $\Symb_p(k)$.

\begin{lem}\label{lem:Symb}
	The map $t_{k,p}^n\colon M_p^n(k)\to H_p^n(k)$ induces an isomorphism 
	$\Symb_p(k) \isomto H^n_{p}(k)$.  
\end{lem}	
\begin{proof}
	By using \autoref{thm:Kato}, 
	we define a map $\varphi_k:H^n_p(k) \to \Symb_p(k)$ by 
	$\varphi_k(\AS{a,b_1,\ldots, b_{n-1}}_k) := \set{a,b_1,\ldots, b_{n-1}}_{k/k}$. 
	To show that $\varphi$ is well-defined, 
	it is left to prove 
	$\set{a,b_1,\ldots,b_{n-1}}_{k/k} = 0$ if $b_i = b_j$ for some $i\neq j$, 
	and  
	$\set{a,a,b_2,\ldots, b_{n-1}}_{k/k} = 0$ in $M_p^n(k)$ for some $a\neq 0$. 
	For the first equality, put $b := b_i = b_j$. We may assume $b\not \in (k^{\times})^p$. 
	Consider the purely inseparable extension $K := k(\sqrt[p]{b})$ 
	of $k$. 
	Since we have $N_{K/k}(\sqrt[p]{b}) = (\sqrt[p]{b})^p = b_i$, 
	the projection formula \PF gives 
	\begin{align*}
		\set{a,b_1,\ldots, b_i,\ldots,b_j,\ldots, b_{n-1}}_{k/k}  
		&= \set{a,b_1,\ldots, N_{K/k}(\sqrt[p]{b}),\ldots, b_j,\ldots, b_{n-1}}_{k/k}  \\
		&= \set{a,b_1,\ldots, \sqrt[p]{b},\ldots, b_j,\ldots, b_{n-1}}_{K/k}.
	\end{align*}
	By $b_j = (\sqrt[p]{b})^p$ in $K$, we obtain $\set{a,b_1,\ldots,b_{n-1}}_{k/k} = 0$ 
	in $M_p^n(k)$.


	{\blue 	To show the equality $\set{a,a,b_2,\ldots,b_{n-1}}_{k/k}= 0$. 
	Consider the polynomial $f(T) = T^p - T - a$ in $k[T]$ 
	and take a root $\alpha\in k^{\sep}$ of $f(T)$.  
	In the case $\alpha\in k$, 
	we derive $\wp(\alpha) = \alpha^p - \alpha = a$, and hence 
	$\set{a,a,b_2,\ldots,b_{n-1}}_{k/k} = 0$ in $M_{p}^n(k)$. 
	When $\alpha \not\in k$, 
	the polynomial $f(T)$ is irreducible. 
	In fact,  
	the remaining roots of $f(T)$ can be expressed as $\alpha + i$ for $i\in \Fp^{\times}$.
	If $g(T)$ is an irreducible factor of degree $r\le p$, 
	then the coefficient of $T^{r-1}$ in $g(T)$ is 
	$-\sum_{j=1}^r(\alpha + i_j)$ for some $i_1,\ldots,i_r\in \Fp$. 
	Since this coefficient belongs to $k$, we conclude $r = p$ implying $f(T) = g(T)$.}
	Now, let us put $K = k(\alpha)$. 
	Then, 
	\[
	-a = f(0) = (-1)^{p}N_{K/k}(\alpha).
	\]
	This yields $a = N_{K/k}(\alpha)$. 
	Thus, we have
	\begin{align*}
	\set{a,a,b_2,\ldots, b_{n-1}}_{k/k}
	 &= \set{a,N_{K/k}(\alpha),b_2,\ldots, b_{n-1}}_{k/k} \\
	&= 	\set{a,\alpha,b_2,\ldots, b_{n-1}}_{K/k}\quad (\mbox{by \PF})\\
	&= \set{\wp(\alpha),\alpha,b_2,\ldots, b_{n-1} }_{K/k}\quad (\mbox{by $\alpha^p - \alpha -a = f(\alpha) = 0$})\\
	&= 0\quad \mbox{in $M_p^n(k)$}.
	\end{align*} 
	Clearly, we have $\varphi_k\circ t_{k,p}^n = \Id$ on $\Symb_p(k)$. 
	The map $t_{k,p}^n$ induces a bijection $\Symb_p(k)\isomto H^n_p(k)$. 
\end{proof}

\begin{prop}\label{prop:BTa}
	Let $k$ be a field of characteristic $p>0$. 
	Suppose that every finite extension of $k$ is of degree $p^r$ for some $r\ge 0$. 
	Then, for $n\ge 2$ and for an extension $K/k$ of degree $p$, 
	{\blue the quotient group $\Omega_K^{n-1}/\mathrm{d}\Omega_K^{n-2}$} is generated by differentials of the forms 
	\begin{align*}
	\xi\dlog y_1\cdots\dlog y_{n-1},\ \mbox{where 
	$\xi \in K$ and $y_1,\ldots, y_{n-1}\in k^{\times}$, and}\\
	x\dlog  \eta \dlog y_2 \cdots  \dlog y_{n-1},\ 
	\mbox{where $x \in k, \eta \in K^{\times}$ and $y_2,\ldots, y_n \in k^{\times}$}. 
	\end{align*}
\end{prop}
\begin{proof}
	Take any 
	$\omega = \alpha\dlog \beta_1 \cdots \dlog \beta_{n-1}$ in $\Omega_K^{n-1}$ 
	for some $\alpha \in K, \beta_1,\ldots, \beta_{n-1}\in K^{\times}$ 
	{\blue with $\alpha\neq 0$.}
	The differential form $\alpha^{-1}\omega = \dlog \beta_1 \cdots \dlog \beta_{n-1}$ 
	is  in the image of the differential symbol map \eqref{eq:ds}
	\[
	\dlog \colon K_{n-1}^M(K)\to \Omega_K^{n-1}; \set{\xi_1,\ldots,\xi_{n-1}}_K \mapsto \dlog \xi_1\cdots \dlog \xi_{n-1}.
	\]
	By \autoref{prop:BT},
	the given $\omega$ is a sum of differentials of the form 
	$\alpha \dlog \beta \dlog y_2 \cdots \dlog y_{n-1}$ for some $\beta\in K^{\times}$ 
	and $y_2,\ldots,y_{n-1}\in k^{\times}$.
	Therefore, the assertion is reduced to the case $n=2$. 
	Namely, it is enough to show that any non-trivial differential form
	$\omega = \alpha \dlog \beta \in {\blue \Omega_K^1/\mathrm{d}K}$ with 
	$\alpha \in K, \beta \in K^{\times}$ is 
	a finite sum of the differentials of the forms 
	\begin{equation}
	\label{eq:O1gen}	
	\xi \dlog y \quad (\xi \in K, y\in k^{\times})\quad  \mbox{and}\quad  
	x\dlog\eta\quad (x\in k^{\times}, \eta \in K^{\times}).
	\end{equation} 
	Write $K = k(\theta) = k[T]/(\pi(T))$ using an irreducible monic polynomial $\pi(T) \in k[T]$ of degree $p$. 
	If $\beta$ is in $k^{\times}$, there is nothing to show. 
	We assume $\beta \not\in k^{\times}$. 
	The element $\beta \in K^{\times}$ is the image of a polynomial $B(T)\in k[T]$ of degree $<p$ 
	by the evaluation map $k[T]\surj k(\theta)= k[T]/(\pi(T))$. 
	Consider the irreducible decomposition $B(T) = y\prod_{i} B_i(T)$ 
	for some $y\in k^{\times}$ and irreducible monic polynomials $B_i(T) \in k[T]$. 
	From the assumption on $k$, 
	the polynomial $B_i(T)$ which has degree $<p$ should be of the form $B_i(T) = T- b_i$ for some $b_i \in k$. 
	Therefore, we have $\beta = y\prod_i(\theta -b_i)$. 
	Write $\alpha = \sum_{j=0}^{p-1}a_j\theta^j$ with $a_j\in k$. 
	The differential form $\omega$ is further decomposed as a sum of differentials of the forms  
	\[
	\alpha \dlog y,\quad \mbox{and}\quad 
	\omega_j := a\theta^j\dlog(\theta-b)
	\]
	with $a,b\in k$.
	The former differential is of the required form \eqref{eq:O1gen}. 
	For the later $\omega_j$, 
	if $j = 0$ or $a = 0$, the assertion holds. 
	For $0 < j<p$ with $a\neq 0$, 
	put $\tau:= \theta -b$. 
	Expanding $\theta^j = (\tau+b)^j$, 
	the differential form $\omega_j$ is a sum of   
	$\omega' := y\tau^i\dlog\tau$ 
	for some $y \in k^{\times}$ and $0\le i <p$. 
	The assertion follows in the case $i=0$. 
	For $i>0$, 
	we have 
	{\blue $i\omega' = y\tau^i\dlog(\tau^i) =  - y\tau^i\dlog y$ modulo $\mathrm{d} K$, 
	because $0 = \mathrm{d}(y\tau^i) = y\tau^i \dlog (y\tau^i) = y\tau^i \dlog y + y\tau^i \dlog(\tau^i)$ in $\Omega_K^{1}/\mathrm{d}K$}.
	For this reason, any differential $\omega = \alpha \dlog \beta$ modulo $dK$ 
	is written by the differentials of the forms given in \eqref{eq:O1gen}.
\end{proof}

\begin{thm}\label{thm:modp}
	For any $n\ge 1$, the map 
	$t_{k,p}^n\colon M_p^n(k)\isomto H_p^n(k)$ is bijective. 
\end{thm}

\begin{proof}
	By \autoref{lem:Symb}, 
	it is enough to show that 
	the map $t_k = t_{k,p}^n\colon M_p^n(k)\to H^n_p(k)$ is injective or 
	$\varphi_k\colon H^n_p(k) \to M_p^n(k)$ 
	which is given by $\varphi_k(\AS{a,b_1,\ldots, b_{n-1}}_k) = \set{a,b_1,\ldots, b_{n-1}}_{k/k}$ 
	is surjective. 
	First, 
	let $k^{(p)}/k$ be the field extension 
	corresponding to the $p$-Sylow subgroup of $\Gal(k^{\sep}/k)$. 
	The following diagram is commutative:
	\[
	\xymatrix{
	M_p^n(k^{(p)}) \ar[r]^-{t_{k^{(p)}}} & H^n_p(k^{(p)}) \\
	M_p^n(k)\ar[u]^{\res_{k^{(p)}/k}} \ar[r]^-{t_k} & H^n_p(k)\ar[u]_{\res_{k^{(p)}/k}}. 
	}
	\]
	Recall that 
	$H^n_p(k)$ is a quotient of $\Omega_k^{n-1}$. 
	Since the correspondence 
	$\Omega_{-}^{n-1}\colon K \mapsto \Omega_K^{n-1}$ 
	forms a Mackey functor, 
	the commutativity of the above diagram comes from the definition of the restriction map 
	of the Mackey product $\Ga\otimesM\Gm^{\otimesM (n-1)}$.
	Furthermore, 
	for any subextension $k \subset K \subset k^{(p)}$ with $[K:k]<\infty$, 
	the extension degree $[K:k]$ is prime to $p$. 
	Therefore,  $\tr_{K/k}\circ \res_{K/k} = [K:k]$ implies that  
	$\res_{K/k}\colon M_p^n(k) \to M_p^n(K)$ is injective. 
	By $M_p^n(k^{(p)}) \simeq \varinjlim_{K}M_p^n(K)$ (for the proof, see \cite[Lem.~3.6.4]{Akh00}), 
	the left vertical map $\res_{k^{(p)}/k}$ in the above diagram is injective. 
	To show that $t_k$ is injective, take any  $\xi\in M_p^n(k)$ 
	with $t_k(\xi) = 0$. 
	If we assume that $t_{k^{(p)}}$ is bijective, 
	we have $\res_{k^{(p)}/k}(\xi) = 0$ in $M_p^n(k)$ 
	and hence $\xi = 0$. 
	Thus, we may assume $k = k^{(p)}$. 
	In particular,  
	every finite extension of $k$ is of degree $p^r$ for some $r\ge 0$. 
	
	We show that the map $\varphi_k$ is surjective 
	by proving the equality 
	$M_p^n(k) = \Symb_p(k)$. 
	Take any non-zero symbol $\set{\alpha,\beta_1,\ldots,\beta_{n-1}}_{K/k}$ 
	in $M_p^n(k)$ 
	with $[K:k]= p^r$. 
	There exists a tower of fields 
	$k = k_0 \subset k_1 \subset k_2 \subset \cdots \subset k_r= K$
	such that $[k_i:k_{i-1}] = p$. 
	By induction on $r$, 
	we may assume $[K:k] = p$. 
	By \autoref{prop:BTa}, the symbol  $\AS{\alpha, \beta_1, \ldots ,\beta_{n-1}}_K 
	= \alpha\dlog \beta_1\cdots \dlog \beta_{n-1}$ 
	is written as a sum of elements of the forms  
	\begin{align*}	
	&\AS{\xi, y_1,\ldots,y_{n-1}}_K\quad \mbox{for some $\xi \in K,$ $y_1,\ldots, y_{n-1}\in k^{\times}$}, and \\
	&\AS{x, \eta,  y_2,\ldots,  y_{n-1}}_K\quad \mbox{for some $x\in k, y_2,\ldots,y_{n-1} \in k^{\times}$, 
	$\eta\in K^{\times}$}. 
	\end{align*}
	
	From \autoref{lem:Symb}, 
	the symbol 
	$\set{\alpha,\beta_1,\ldots,\beta_{n-1}}_{K/K}$ 
	is a sum of 
	symbols of the forms 
	$\set{\xi,  y_1,\ldots,  y_{n-1}}_{K/K}$  
	and $\set{ x, \eta,  y_2,\ldots,  y_{n-1}}_{K/K}$.
	Hence, 
	\[
	\set{\alpha,\beta_1,\ldots,\beta_{n-1}}_{K/k} 
	= \tr_{K/k}(\set{\alpha,\beta_1,\ldots,\beta_{n-1}}_{K/K})
	\] 
	is a sum of symbols of the form
	$\set{\xi,  y_1,\ldots,  y_{n-1}}_{K/k}$ and 
	$\set{ x, \eta,  y_2,\ldots,  y_{n-1}}_{K/k}$. 
	By the projection formula \PF, we have 
	\begin{align*}
	\set{\xi,  y_1,\ldots,  y_{n-1}}_{K/k} &= 
	\set{\Tr_{K/k}(\xi), y_1,\ldots, y_{n-1}}_{k/k}, \\
	\set{x, \eta,  y_2,\ldots,  y_{n-1}}_{K/k} &= 
	\set{x, N_{K/k}(\eta), y_2,\ldots, y_{n-1}}_{k/k}.
	\end{align*}
	Because of this, 
	$\set{\alpha,\beta_1,\ldots,\beta_{n-1}}_{K/k}$ 
	is in $\Symb_p(k)$ 
	and the map $\varphi_k$ is surjective. 
\end{proof}

\begin{thm}
\label{thm:bij}
	Let $k$ be an arbitrary field of characteristic $p>0$. 
	For $r\ge 1$ and $n\ge 1$, 
	the map 
	$t_{k,p^r}^n\colon M_{p^r}^n(k) \isomto H^n_{p^r}(k)$
	is bijective. 
\end{thm}
\begin{proof}
	We show the assertion by induction on $r$. 	
	The case $r=1$ follows from \autoref{thm:modp}. 
	For $r\ge 1$, we assume that $t_{k,p^r}^n$ is bijective. 
	The short exact sequence $0 \to W_r \to W_{r+1} \to  \Ga \to 0$ 
	induces a commutative diagram with exact rows: 
	\[
	\xymatrix{
	(W_r\otimesM\Gm^{\otimesM (n-1)})(k) \ar[r]\ar[d]^{\wp} & (W_{r+1}\otimesM\Gm^{\otimesM (n-1)})(k)\ar[r]\ar[d]^{\wp} &  (\Ga \otimesM\Gm^{\otimesM (n-1)})(k)\ar[r]\ar[d]^{\wp} & 0\,\\
	(W_r\otimesM\Gm^{\otimesM (n-1)})(k) \ar[r] & (W_{r+1}\otimesM\Gm^{\otimesM (n-1)})(k)\ar[r] &  (\Ga \otimesM\Gm^{\otimesM (n-1)})(k)\ar[r] & 0.
	}
	\]
	The diagram above induces  the following commutative diagram:
	\begin{equation}\label{eq:Mnpr}
		\vcenter{
		\xymatrix{
		  & M_{p^r}^n(k)\ar[d]^{t_{p^r}^n}_{\simeq} \ar[r] & M_{p^{r+1}}^n(k) \ar[r]\ar@{->>}[d]^{t_{p^{r+1}}^n} & M_{p}^n(k)\,\ar[d]^{t_{p}^n}_{\simeq}  \\
		 H^{n-1}(k,\Z/p(n-1))\ar[r]^-{\delta} & H_{p^r}^n(k) \ar[r] & H_{p^{r+1}}^n(k) \ar[r] & H^n_{p}(k).
		}}
	\end{equation}
	Here, 
	the bottom sequence is a part of the long exact sequence induced from 
	the short exact sequence 
	\[
	0 \to W_r\Omega_{k^{\sep},\log}^{n-1} \to W_{r+1}\Omega_{k^{\sep},\log}^{n-1} \to \Omega_{k^{\sep},\log}^{n-1} \to 0
	\]
	(\cite[Chap.~I, Thm.~5.7.2]{Ill79}, {\blue \cite[Prop.~2.12]{Shi07}}). 
	By the inductive hypothesis, the vertical maps $t_{k,p^r}^n$ and $t_{k,p}^n$ are bijective.
	The lower horizontal sequence in the above diagram \eqref{eq:Mnpr} 
	continues further to the left as 
	\[
	\xymatrix{
	K_{n-1}^M(k)/p^{r+1} \ar[d]^{s_{k,p^{r+1}}^{n-1}}_{\simeq}\ar@{->>}[r]^{\bmod p} & K_{n-1}^M(k)/p\ar[d]^{s_{k,p}^{n-1}}_{\simeq}  \\
	H^{n-1}(k,\Z/p^{r+1}(n-1)) \ar[r] & H^{n-1}(k,\Z/p(n-1))\ar[r]^-{\delta} & H^{n}_{p^r}(k).
	}
	\]
	By \autoref{thm:BKG}, the vertical maps in the above diagram are bijective 
	and hence $\delta$ is the $0$ map. 
	Thus, $t_{k,p^{r+1}}^n$ is injective by the diagram chase in \eqref{eq:Mnpr}. 
\end{proof}

For a truncation set $S$, 
we have a homomorphism $\wp = \WS(\phi) -\Id\colon \WS(K)\to \WS(K)$  
and this induces a map 
{\blue $\wp \colon \W_S\otimesM \Gm^{\otimesM (n-1)}\to \WS \otimesM \Gm^{\otimesM (n-1)}$ 
of Mackey functors} as in \eqref{eq:wpM}. 
We denote by 
$M_S^n := (\W_S\otimesM \Gm^{\otimesM (n-1)})/\wp$ the cokernel of $\wp$. 
When $S$ is finite,  
{\blue we have $\WS(K) \simeq \prod_{(m,p)=1}W_{r_m}(K)$ 
	for any field extension $K/k$  
	as in \eqref{eq:dec}. 
	This gives the decomposition $\WS = \bigoplus_{(m,p)=1}W_{r_m}$ 
	as Mackey functors.}
Since the Mackey product commutes with the direct sum, 
we have 
\begin{equation}\label{eq:decM}
	M_S^n(k) \simeq \prod_{(m,p)=1} M_{p^{r_m}}^n(k).
\end{equation}
According to \eqref{eq:strH^n}, we define 
\[
H^n_S(k) := \Coker\left(\WSOmega_k^{n-1}\stackrel{\eqref{eq:dec}}{\simeq} \prod_{(m,p) = 1} W_{r_m}\Omega_k^n \xrightarrow{\prod \wp}  \prod_{(m,p) = 1}W_{r_m}\Omega_k^{n-1}/ \mathrm{d}V^{r_m-1} \Omega_k^{n-2}\right).
\]
\autoref{thm:bij} extends to any finite truncation set as follows:

\begin{cor}
\label{cor:S0}
	Let $k$ be a field of characteristic $p>0$. 
	For any finite truncation set $S$ and $n\ge 1$, we have an isomorphism 
	$M_S^n(k) \simeq H^n_S(k)$.
\end{cor}

\begin{cor}
	Let $k$ be a field of characteristic $p>0$. 
	For a truncation set $S$ and $n\ge 1$, we have an isomorphism 
	\[
	M_S^n(k) \simeq \varprojlim_{S_0\subset S} H^n_{S_0}(k),
	\] 	
	where $S_0$ runs the set of all finite truncation set contained in $S$.
\end{cor}
\begin{proof}
	By the definition, we have $\W_S(K) = \varprojlim_{S_0}\W_{S_0}(K)$. 
	Since the restriction map $R_{S_0'}^{S_0}\colon \W_{S_0}(K)\to \W_{S_0'}(K)$ 
	is surjective, for any finite truncation sets $S_0'\subset S_0$, 
	the Mittag-Leffler condition holds. 
	We obtain 
	$(\W_S\otimesM \Gm^{\otimesM (n-1)})(k)/\wp \simeq  \varprojlim_{S_0}(\W_{S_0}\otimesM \Gm^{\otimesM (n-1)})(k)/\wp$. 
	The assertion follows from 
	\autoref{cor:S0}.
\end{proof}

\begin{cor}
	Let $k$ be a field of characteristic $p>0$. 
	For any finite truncation set $S$ and $n\ge 1$, we have an isomorphism
	\[
	(\WS\otimesM K_{n-1}^M)(k)/\wp \simeq H^n_S(k),
	\]
	where 
	$\wp\colon (\WS\otimesM K_{n-1}^M)(k) \to (\WS\otimesM K_{n-1}^M)(k)$ is
	given by 
	$\wp(\set{\mathbf{a},\mathbf{b}}_{K/k}) = \set{\wp(\mathbf{a}),\mathbf{b}}_{K/k}$. 
\end{cor}
\begin{proof}
	By the decomposition $\WS = \bigoplus_{(m,p)=1}W_{r_m}$ 
	as Mackey functors, we have 
	\begin{align*}
	(\WS\otimesM K_{n-1}^M)(k)/\wp  &\simeq \bigoplus_{(p,m)=1} (W_{r_m} \otimesM K_{n-1}^M)(k)/\wp\\
	&\simeq \bigoplus_{(m,p)=1} (W_{r_m}/\wp \otimesM K_{n-1}^M/p^{r_m})(k)\quad (\mbox{by \autoref{lem:mod}})\\
	&\simeq \bigoplus_{(m,p)=1} (W_{r_m}/\wp \otimesM \Gm^{\otimesM (n-1)}/p^{r_m})(k)\quad (\mbox{by \autoref{thm:bij}})\\
	&\simeq \bigoplus_{(m,p)=1} (W_{r_m} \otimesM \Gm^{\otimesM (n-1)})(k)/\wp\quad (\mbox{by \autoref{lem:mod}})\\
	&= M^n_S(k).
	\end{align*}
	\autoref{cor:S0} gives the assertion.	
\end{proof}


\def\cprime{$'$}
\providecommand{\bysame}{\leavevmode\hbox to3em{\hrulefill}\thinspace}
\providecommand{\href}[2]{#2}

\end{document}